\documentclass[12pt]{article}
\usepackage{amsmath,amssymb}

\newtheorem{theorem}{Theorem}[section]
\newtheorem{proposition}[theorem]{Proposition}

\newtheorem{example}[theorem]{Example}
\newtheorem{remark}[theorem]{Remark}

\newenvironment{proof}{\smallskip\par{\sc Proof.}\enspace}%
 {{\unskip\nobreak\hfil\penalty50\hskip2em
          \hbox{}\nobreak\hfil{\rule[-1pt]{5pt}{10pt}}
          \parfillskip=0pt\finalhyphendemerits=0
          \par\medskip}} 

\makeatletter
\def\section{\@startsection {section}{1}{\z@}{3.25ex plus 1ex minus
 .2ex}{1.5ex plus .2ex}{\large\bf}}
\def\subsection{\@startsection{subsection}{2}{\z@}{3.25ex plus 1ex minus
 .2ex}{1.5ex plus .2ex}{\normalsize\bf}}
\@addtoreset{equation}{section} 
\makeatother

\title{Absolute continuity and Fokker-Planck equation for the law of Wong-Zakai approximations of It\^o's stochastic differential equations}

\author{Alberto Lanconelli\footnote{Dipartimento di Scienze Statistiche "Paolo Fortunati", Alma Mater Studiorum Università di Bologna, Via Belle Arti 41 - 40126 Bologna - Italia. E-mail: \emph{alberto.lanconelli2@unibo.it}}}

\date{\empty}

\begin{document}

\maketitle

\numberwithin{equation}{section}

\bigskip

\begin{abstract}
We investigate the regularity of the law of Wong-Zakai-type approximations for It\^o stochastic differential equations. These approximations solve random differential equations where the diffusion coefficient is Wick-multiplied by the smoothed white noise. Using a criteria based on the Malliavin calculus we establish absolute continuity and a Fokker-Planck-type equation solved in the distributional sense by the density. The parabolic smoothing effect typical of the solutions of It\^o equations is lacking in this approximated framework; therefore, in order to prove absolute continuity, the initial condition of the random differential equation needs to possess a density itself.    
\end{abstract}



\allowdisplaybreaks

\section{Introduction and statement of the main result}

The celebrated Wong-Zakai theorem \cite{WZ},\cite{WZ2}, extended to the multidimensional case by Stroock and Varadhan \cite{Stroock Varadhan}, provides a crucial insight in the theory of stochastic differential equations. It asserts that, given a suitable smooth approximation $\{B_t^{\varepsilon}\}_{t\in [0,T]}$ of the Brownian motion $\{B_t\}_{t\in [0,T]}$, the solution $\{X_t^{\varepsilon}\}_{t\in [0,T]}$ of the random ordinary differential equation
\begin{eqnarray}\label{stra approx intro}
\dot{X}_t^{\varepsilon}=b(t,X_t^{\varepsilon})+\sigma(t,X_t^{\varepsilon})\cdot\dot{B}_t^{\varepsilon},
\end{eqnarray}
converges, as $\varepsilon$ goes to zero, to the solution of the Stratonovich stochastic differential equation (SDE, for short)
\begin{eqnarray*}\label{stra intro}
dX_t=b(t,X_t)dt+\sigma(t,X_t)\circ dB_t
\end{eqnarray*}
instead to the more popular It\^o's interpretation of the corresponding stochastic equation, 
\begin{eqnarray*}\label{ito intro}
dX_t=b(t,X_t)dt+\sigma(t,X_t)dB_t.
\end{eqnarray*}
The understanding behind this phenomenon is not the way the Brownian motion $\{B_t\}_{t\in [0,T]}$ is approximated but rather the way we multiply the diffusion coefficient $\sigma(t,X_t^{\varepsilon})$ with the smoothed white noise $\dot{B}_t^{\varepsilon}$ in (\ref{stra approx intro}). The following example clarifies this point.  
\begin{example}
Consider the random ordinary differential equation
\begin{eqnarray*}
\dot{X}^{\varepsilon}_t=X^{\varepsilon}_t\cdot\dot{B}^{\varepsilon}_t,\quad X^{\varepsilon}_0=x
\end{eqnarray*}
which corresponds to (\ref{stra approx intro}) for $b(t,x)=0$ and $\sigma(t,x)=x$. Its solution
\begin{eqnarray*}
X^{\varepsilon}_t=x\exp\{B^{\varepsilon}_t\},\quad t\in [0,T]
\end{eqnarray*}
converges to 
\begin{eqnarray}\label{1}
X_t:=x\exp\{B_t\},\quad t\in [0,T]
\end{eqnarray}
whenever for each $t\in [0,T]$ the random variable $B_t^{\varepsilon}$ converges to $B_t$ in probability as $\varepsilon$ tends to zero. A direct verification shows that (\ref{1}) is the unique solution of  
\begin{eqnarray*}
dX_t=X_t/2dt+X_tdB_t,\quad X_0=x
\end{eqnarray*}
which is equivalent to 
\begin{eqnarray*}
dX_t=X_t\circ dB_t,\quad X_0=x.
\end{eqnarray*}
\end{example}
In fact, it is well known (see for instance Karatzas and Shreve \cite{KS}) that the Stratonovich SDE
\begin{eqnarray*}
dX_t=b(t,X_t)dt+\sigma(t,X_t)\circ dB_t,\quad X_0=x
\end{eqnarray*}
is equivalent to the It\^o SDE
\begin{eqnarray*}
dX_t=\Big[b(t,X_t)+\sigma(t,X_t)\partial_x\sigma(t,X_t)/2\Big]dt+\sigma(t,X_t)dB_t,\quad X_0=x.
\end{eqnarray*}
Therefore, in order to recover the It\^o interpretation of the limiting SDE one may start with the modified equation
\begin{eqnarray*}
\dot{Y}^{\varepsilon}_t=\Big[b(t,Y^{\varepsilon}_t)-\sigma(t,Y^{\varepsilon}_t)\partial_y\sigma(t,Y^{\varepsilon}_t)/2\Big]+\sigma(t,Y^{\varepsilon}_t)\cdot\dot{B}^{\varepsilon}_t,\quad Y^{\varepsilon}_0=x
\end{eqnarray*}
to obtain in the limit
\begin{eqnarray*}
dY_t=\Big[b(t,Y_t)-\sigma(t,Y_t)\partial_x\sigma(t,Y_t)/2\Big]dt+\sigma(t,Y_t)\circ dB_t,\quad Y_0=x
\end{eqnarray*}
which corresponds to
\begin{eqnarray*}
dY_t=b(t,Y_t)dt+\sigma(t,Y_t)dB_t,\quad Y_0=x.
\end{eqnarray*}
However, this procedure has some drawbacks. For instance, certain probabilistic properties of the exact solution may be lost in the approximated solution.
\begin{example}\label{example 1}
Suppose we wish to approximate
\begin{eqnarray}\label{A}
dY_t=Y_tdB_t,\quad Y_0=x
\end{eqnarray}
according to the previous procedure; then, we should consider the random ordinary differential equation
\begin{eqnarray*}
\dot{Y}^{\varepsilon}_t=-Y^{\varepsilon}_t/2+Y^{\varepsilon}_t\cdot\dot{B}^{\varepsilon}_t,\quad Y^{\varepsilon}_0=x
\end{eqnarray*}
whose solution is
\begin{eqnarray*}
Y_t^{\varepsilon}=x\exp\left\{B_t^{\varepsilon}-t/2\right\}, \quad t\in[0,T].
\end{eqnarray*}
However, on one hand we have
\begin{eqnarray*}
E[Y_t]=x\quad\mbox{ for all $t\in [0,T]$}
\end{eqnarray*}
while on the other
\begin{eqnarray}\label{B}
E[Y^{\varepsilon}_t]\neq x\quad\mbox{ unless }\quad E[(B_t^{\varepsilon})^2]=t\mbox{ for all $\varepsilon>0$}.
\end{eqnarray}
\end{example}
The problem of finding a version of equation (\ref{stra approx intro}) having in the limit the It\^o interpretation of the SDE (by-passing the Stratonovich interpretation involved in the procedure described above) was partially solved by Hu and {\O}ksendal \cite{HO}: they proved that the solution of
\begin{eqnarray}\label{ito approx intro}
\dot{Y}_t^{\varepsilon}=b(t,Y_t^{\varepsilon})+\sigma(t)Y_t^{\varepsilon}\diamond\dot{B}_t^{\varepsilon},
\end{eqnarray}
where $\diamond$ stands for the Wick product, converges as $\varepsilon$ goes to zero to the solution of the It\^o SDE
\begin{eqnarray}\label{ito oksendal intro}
dY_t=b(t,Y_t)dt+\sigma(t)Y_tdB_t.
\end{eqnarray}
Here $\sigma$ is a deterministic function; the assumption of a linear diffusion coefficient is utilized in connection with a reduction method to solve equation (\ref{ito approx intro}).  
\begin{example}
Referring to Example \ref{example 1}, we now utilize (\ref{ito approx intro}) to approximate equation (\ref{A}), i.e. 
\begin{eqnarray*}
\dot{Y}^{\varepsilon}_t=Y^{\varepsilon}_t\diamond\dot{B}^{\varepsilon}_t,\quad Y^{\varepsilon}_0=x.
\end{eqnarray*}
The solution to the previous equation can be computed using the Wick calculus as
\begin{eqnarray*}
Y^{\varepsilon}_t=x\exp^{\diamond}\{B^{\varepsilon}_t\}=x\exp\left\{ B^{\varepsilon}_t-E[|B^{\varepsilon}_t|^2]/2\right\}.
\end{eqnarray*}
Therefore,
\begin{eqnarray*}
\lim_{\varepsilon\to 0}Y^{\varepsilon}_t=x\exp\left\{B_t-t/2\right\}\quad\mbox{ as long as }\quad B_t^{\varepsilon}\to B_t \mbox{ in $L^2(\Omega)$}
\end{eqnarray*}
which is the solution of equation (\ref{A}). Moreover, in this case
\begin{eqnarray*}
E[Y_t^{\varepsilon}]=E[Y_t]=x\quad\mbox{ for any $t\in [0,T]$}
\end{eqnarray*}
independently of the particular approximation of $\{B_t\}_{t\in [0,T]}$ utilized (in contrast with (\ref{B})).
\end{example}
Roughly speaking, the reason for choosing equation (\ref{ito approx intro}) as the correct approximated version of the It\^o SDE (\ref{ito oksendal intro}) can be traced back to the identity
\begin{eqnarray}\label{core}
\int_0^T\gamma_tdB_t=\int_0^T\gamma_t\diamond\dot{B}_tdt
\end{eqnarray}
which connects the It\^o-Skorohod integral of a stochastic process $\{\gamma_t\}_{t\in [0,T]}$ (on the left hand side) with a \emph{standard} integral of the Wick product between $\gamma_t$ and the white noise $\dot{B}_t$. We refer the reader to the book Holden et al. \cite{HOUZ} for a proof of this identity. \\

The aim of the present paper is to investigate the probabilistic properties of the Wong-Zakai-type approximation (\ref{ito approx intro}), as suggested in \cite{HO}, of the It\^o SDE (\ref{ito oksendal intro}). In particular, we will focus on the absolute continuity of the law of the solution $Y_t^{\varepsilon}$ finding sufficient conditions for the existence of a density with respect to the one dimensional Lebesgue measure. We will also write a Fokker-Planck-type equation solved in the distributional sense by the density of $Y_t^{\varepsilon}$. This analogy with exact solutions of It\^o SDEs is however conditioned by a restriction on the initial data of equation (\ref{ito approx intro}) that is required to possess a density. Our approach relies on the criteria for absolute continuity based on the Malliavin calculus. We stress that Wong-Zakai approximations, contrary to many other approximation schemes like for instance the Euler discretization (see e.g. Bally and Talay \cite{BT} for an investigation on the convergence of the densities of such approximations), are defined in terms of non adapted stochastic processes.\\
There is a vast literature on Wong-Zakai approximations for Stratonovich stochastic (partial) differential equations driven by different types of noise: one may look at Brezniak and Flandoli \cite{Brezniak Flandoli}, Gy\"ongy and A. Shmatkov \cite{Gyongy Shmatkov}, Hu et al. \cite{HKX}, Hu and Nualart \cite{HN}, Konecny \cite{Konecny}, Tessitore and Zabczyk \cite{Tessitore Zabczyk} just to mention a few. We also mention the remarkable paper Hairer and Pardoux \cite{Hairer Pardoux} where a Wong-Zakai theorem for a general nonlinear It\^o-type stochastic heat equation driven by a space-time white noise is proved.\\
Wong-Zakai approximations for It\^o SDEs are quite rare in the literature. As the insight of \cite{HO} shows, one has to deal in this case with equations involving the Wick product, which corresponds to treat Skorohod SDEs (see Section \ref{proofs}). This type of equations possesses a global solution only in some particular cases making even the existence of the Wong-Zakai approximation a tough issue.  We mention the paper Da Pelo et al. \cite{DLS} dealing with Wong-Zakai approximations for It\^o-Stratonovich interpolations and Ben Ammou and Lanconelli \cite{BL} investigating the rate of convergence for Wong-Zakai approximations of It\^o SDEs in the spirit of the present paper.\\

\noindent To state our main result we need to introduce a few notation. Let $(\Omega,\mathcal{F},P)$ be the classical Wiener space over the time interval $[0,T]$, where $T$ is an arbitrary positive constant, and denote by $\{B_t\}_{t\in [0,T]}$ the coordinate process, i.e.
\begin{eqnarray*}
B_t:\Omega&\to&\mathbb{R}\\
\omega&\mapsto& B_t(\omega)=\omega(t).
\end{eqnarray*}
By construction, the process $\{B_t\}_{t\in [0,T]}$ is under the measure $P$ a one dimensional Brownian motion. Now, let $\pi$ be a finite partition of the interval $[0,T]$, that means $\pi=\{t_0,t_1,...,t_{n-1},t_n\}$ with
\begin{eqnarray*}
0=t_0<t_1<\cdot\cdot\cdot<t_{n-1}<t_n=T
\end{eqnarray*}
and consider the \emph{polygonal} approximation of the Brownian motion $\{B_t\}_{t\in [0,T]}$ relative to the partition $\pi$:
\begin{eqnarray}\label{polygonal}
B_t^{\pi}:=\left(1-\frac{t-t_k}{t_{k+1}-t_k}\right)B_{t_{k}}+\frac{t-t_k}{t_{k+1}-t_k}B_{t_{k+1}}\quad\mbox{if }t\in [t_{k},t_{k+1}[
\end{eqnarray}
and $B^\pi_T:=B_T$. It is well known that for any $\varepsilon>0$ and $p\geq 1$ there exists a positive constant $C_{p,T,\varepsilon}$ such that
\begin{eqnarray*}
\left(E\left[\sup_{t\in [0,T]}|B_t^{\pi}-B_t|^p\right]\right)^{1/p}\leq C_{p,T,\varepsilon}|\pi|^{1/2-\varepsilon}
\end{eqnarray*}
where $|\pi|:=\max_{k\in\{0,...,n-1\}}(t_{k+1}-t_k)$ stands for the mesh of the partition $\pi$. We refer the reader to Lemma 2.1 in Hu et al. \cite{HKX} and Lemma 11.8 in Hu \cite{Hu} for sharper estimates. We assume that the finite partition $\pi$ is fixed throughout the present paper. \\
Given the partition $\pi$, let $h^{\pi}$ be a function in $L^2([0,T])$ such that $h^{\pi}\neq 0$ almost everywhere and
\begin{eqnarray*}
\frac{1}{t_{k+1}-t_k}\int_{t_k}^{t_{k+1}}h^{\pi}(u)du=0\quad\mbox{ for all }k\in\{0,..., n-1\}.
\end{eqnarray*}
The crucial role of the function $h^{\pi}$ will be made clear in Proposition \ref{direction} below. The following is our main result: the symbols $\diamond$, $\mathbb{D}^{2,p}$ and $D_{h^{\pi}}$ denote the Wick product, the Sobolev-Malliavin space and the directional Malliavin derivative in the direction $h^{\pi}$, respectively. Definitions and useful properties are postponed to Section \ref{proofs}.

\begin{theorem}\label{main theorem}
Let $\{X^{\pi}_t\}_{t\in [s,T]}$ be the unique solution of the random Cauchy problem
\begin{eqnarray}
\left\{ \begin{array}{ll}\label{SDE}
\dot{X}_t^{\pi}= b(t,X_t^{\pi})+\sigma(t)X_t^{\pi}\diamond \dot{B}_t^{\pi},\quad t\in]s,T] \\
X_s^{\pi}= Y
\end{array}\right.
\end{eqnarray}
where $s\in [0,T[$ and we assume that
\begin{itemize}
\item $b:[0,T]\times\mathbb{R}\to\mathbb{R}$ is a continuous function with bounded first and second partial derivatives with respect to the second variable;
\item $\sigma:[0,T]\to\mathbb{R}$ belongs to $L^2([0,T])$;
\item $Y\in\mathbb{D}^{2,p}$ for all $p\geq 1$ and $E[|D_{h^{\pi}}Y|^{-q}]$ is finite for some $q>4$.
\end{itemize}
Then, 
\begin{enumerate}
\item for any $t\in [s,T]$ the law of $X_t^{\pi}$ is absolutely continuous with respect to the one dimensional Lebesgue measure with a bounded and continuous density;
\item the density $(t,x)\mapsto p^{\pi}(t,x)$ of the random variable $X_t^{\pi}$ solves in the sense of distributions the Fokker-Planck equation $(\partial_t+\mathcal{L}_x)^{\ast}u(t,x)=0$ where $(\partial_t+\mathcal{L}_x)^{\ast}$ stands for the formal adjoint of the operator $\partial_t+\mathcal{L}_x$ and
\begin{eqnarray}\label{PDE}
\mathcal{L}_x:=b(t,x)\partial_x+\sigma(t)xg(t,x)\partial_{xx}.
\end{eqnarray}  
for a suitable measurable function $g:[s,T]\times\mathbb{R}\to\mathbb{R}$. Moreover,
\begin{eqnarray*}
\int_{\mathbb{R}}|g(t,x)|^qp^{\pi}(t,x)dx\mbox{ is finite for all $q\geq 1$ and $t\in [s,T]$}.
\end{eqnarray*}
\end{enumerate}
\end{theorem}

\begin{remark}
We observe that  the assumptions on $Y$ entail the absolute continuity of its law with respect to the one dimensional Lebesgue measure. In fact, contrary to the assumptions usually adopted for the study of the absolute continuity for exact solutions of SDEs, here we need the initial condition to have a density. The stochastic equation (\ref{SDE}), being an approximated version of an It\^o SDE, does not possess the same smoothing properties of the original equation. 
\end{remark}

\begin{remark}
The way we prove the existence of the function $g$ appearing in the operator $\mathcal{L}_x$ does not give us information about its regularity and sign (see formula (\ref{definition of g}) below). Therefore, we don't know whether the existence of a density for the law of $X_t^{\pi}$ can be deduced from the properties of the operator $\mathcal{L}_x$ in (\ref{PDE}). From this point of view, the criteria for absolute continuity based on the Malliavin calculus as utilized in this paper turns out to be crucial. 
\end{remark}

The paper is organized as follows: in Section \ref{proofs} we begin recalling definitions and auxiliary results from the Malliavin calculus and analysis on the Wiener space which will be employed in the proof of Theorem \ref{main theorem}; Section \ref{subsection 1} deals with the proof of the absolute continuity of the law of the Wong-Zakai approximation while Section \ref{subsection 2} addresses the derivation of the Fokker-Planck-type equation solved by the density. Here, an anticipating It\^o formula for the Wong-Zakai approximation (see Theorem \ref{Ito formula} below) plays a major role. An example illustrating the results previously obtained closes the section. 

\section{Proof of Theorem \ref{main theorem}}\label{proofs}

The proof of our main theorem will be divided in two parts: the existence of the density in Section \ref{subsection 1} and the Fokker-Planck equation in Section \ref{subsection 2}. We first set the notation and recall few auxiliary results from the Malliavin Calculus. For more details we refer the reader to one of the books Bogachev \cite{Bogachev}, Hu \cite{Hu}, Janson \cite{Janson} and Nualart \cite{Nualart}. Here we adopt the presentation of \cite{Nualart}.\\

\noindent Let $\mathcal{S}$ denote the class of \emph{smooth} random variables of the form
\begin{eqnarray}\label{smooth r.v.}
F=\varphi\left(\int_0^Th_1(u)dB_u,..., \int_0^Th_n(u)dB_u\right)
\end{eqnarray}
where $\varphi\in C_{\mathcal{P}}^{\infty}(\mathbb{R}^n)$ (the space of infinitely differentiable functions having, together with all their partial derivatives, polynomial growth), the functions $h_1$, ...,$h_n$ are elements of $L^2([0,T])$ and $n\geq 1$ . In the sequel to ease the notation we will set $H:=L^2([0,T])$ and denote by $\Vert\cdot\Vert_p$ the norm in $L^p(\Omega)$. We note that $\mathcal{S}$ is dense in $L^p(\Omega)$ for all $p\geq 1$. \\
The \emph{derivative} of a smooth random variable $F$ of the form (\ref{smooth r.v.}) is the $H$-valued random variable $DF$ given by
\begin{eqnarray*}
D_tF=\sum_{i=1}^n\partial_i \varphi\left(\int_0^Th_1(u)dB_u,..., \int_0^Th_n(u)dB_u\right)h_i(t),\quad t\in [0,T].
\end{eqnarray*}
It is easy to see that for any $h\in H$ and smooth random variables $F$ and $G$ we have the following integration-by-parts formula 
\begin{eqnarray*}
E[GD_hF] = E\left[-FD_hG + FG\int_0^Th(t)dB_t\right].
\end{eqnarray*}
Here $E[\cdot]$ denotes the expectation on the probability space $(\Omega,\mathcal{F}, P)$ while $D_hZ:=\int_0^TD_uZ\cdot h(u)du$ stands for the \emph{directional derivative} of $Z$ in the direction $h\in H$.\\
By means of the previous identity one can prove that the operator $D$ is closable from $L^p(\Omega)$ to $L^p(\Omega;H)$
for any $p\geq 1$. Therefore, one can define the space $\mathbb{D}^{1,p}$ as the closure of $\mathcal{S}$ with
respect to the norm
\begin{eqnarray*}
\Vert F\Vert_{1,p} =\left(E\left[|F|^p\right] + E\left[\vert DF\vert^p_H\right]\right)^{\frac{1}{p}}.
\end{eqnarray*}
Iterating the action of the operator $D$ in such a way that for a smooth random variable $F$ the iterated derivative $D^kF$ is a random variable with values in $H^{\otimes k}$,  we introduce on $\mathcal{S}$ for every $p\geq 1$ and any natural number $k\geq 1$ the seminorm defined by
\begin{eqnarray*}
\Vert F\Vert_{k,p}=\left(E\left[|F|^p\right] +\sum_{j=1}^kE\left[\vert D^j F\vert^p_{H^{\otimes j}}\right]\right)^{\frac{1}{p}}.
\end{eqnarray*}
We will denote by $\mathbb{D}^{k,p}$ the completion of $\mathcal{S}$ with respect to the norm $\Vert\cdot\Vert_{k,p}$.\\

\noindent We now introduce the \emph{divergence operator} $\delta$ which is the adjoint of the operator $D$. It is a closed and unbounded operator on $L^2(\Omega;H)$ with values in $L^2(\Omega)$ such that:
\begin{itemize}
\item the domain of $\delta$ is the set of $H$-valued square integrable random variables $u\in L^2(\Omega;H)$ such that
\begin{eqnarray*}
|E[\langle DF, u\rangle_H]|\leq c\Vert F\Vert_2,
\end{eqnarray*}
for all $F\in\mathbb{D}^{1,2}$, where $c$ is a constant depending only on $u$
\item if $u$ belongs to the domain of $\delta$, then $\delta(u)$ is the element of $L^2(\Omega)$ characterized
by
\begin{eqnarray*}
E[F\delta(u)]=E[\langle DF, u\rangle_H]
\end{eqnarray*}
for any $F\in\mathbb{D}^{1,2}$.
\end{itemize}
One can prove that $\mathbb{D}^{1,2}(H)$ is included in the domain of $\delta$. In particular, if $F\in\mathbb{D}^{1,2}$ and $h\in H$ , then $Fh\in\mathbb{D}^{1,2}(H)$ and
\begin{eqnarray}\label{definition Wick}
F\diamond\int_0^Th(u)dB_u:=\delta(Fh)\in L^2(\Omega)
\end{eqnarray}
is called the \emph{Wick product} of $F$ and $\int_0^Th(u)dB_u$. This definition can be generalized to include more general random variables in the place of $\int_0^Th(u)dB_u$. Moreover, it follows from the properties of $D$ and $\delta$ that
\begin{eqnarray*}
F\diamond\int_0^Th(u)dB_u=F\cdot\int_0^Th(u)dB_u-D_hF.
\end{eqnarray*}
\begin{remark}
There are different ways to define the Wick product of two random variables: via the Wiener-It\^o chaos expansion or specifying its action on stochastic exponentials (see below) or through the so-called $S$-transform. Here, we preferred to use (\ref{definition Wick}) since it clearly shows the duality between Wick product and Malliavin derivative. Other approaches can be found in \cite{HOUZ} and \cite{Hu}. We also remark that, in addition to its key role in the theory of It\^o-Skorohod integration (\ref{core}), the Wick product has important probabilistic interpretations in the Gaussian and Poissonian analysis. See \cite{LS}, \cite{LS Bernoulli}, \cite{L} and the references quoted there.    
\end{remark}
For $h\in H$ we define the \emph{stochastic exponential} to be a random variable of the form
\begin{eqnarray*}
\mathcal{E}(h):=\exp\left\{\int_0^Th(u)dB_u-\frac{1}{2}\int_0^Th^2(u)du\right\}.
\end{eqnarray*}
We remark that the span of such family of elements is dense in $L^p(\Omega)$ and $\mathbb{D}^{k,p}$ for any $p\geq 1$ and $k\in\mathbb{N}$. For $g\in H$ we also define the \emph{translation operator} $\mathcal{T}_g$ as the operator that shifts the Brownian path by the function $\int_0^{\cdot}g(u)du$; more precisely, the action of $\mathcal{T}_g$ on a stochastic exponential is given by
\begin{eqnarray*}
\mathcal{T}_g\mathcal{E}(h):=\mathcal{E}(h)\cdot\exp\{\langle h,g\rangle_H\}.
\end{eqnarray*}
There is a close relationship between Wick product and translation operators; it is the so called Gjessing formula:
\begin{eqnarray}\label{Gjessing}
F\diamond\mathcal{E}(h)=\mathcal{T}_{-h}X\cdot\mathcal{E}(h)
\end{eqnarray} 
which is valid for any $h\in H$ and $F\in L^p(\Omega)$ for some $p> 1$. We refer to Holden et al. \cite{HOUZ} and Janson \cite{Janson} for more details on Wick product and translation operators.\\

\noindent We complete this preliminary part observing that the polygonal approximation $\{B_t^{\pi}\}_{t\in [0,T]}$ of the Brownian motion $\{B_t\}_{t\in [0,T]}$ defined in (\ref{polygonal}) can be written also in the compact form
\begin{eqnarray}\label{B smooth}
B_t^{\pi}=\int_0^TK_t^{\pi}(u)dB_u
\end{eqnarray} 
with
\begin{eqnarray}\label{def K}
K_t^{\pi}(u):=\sum_{k=0}^{n-1}\left(1_{[0,t_k[}(u)+\frac{t-t_k}{t_{k+1}-t_k}1_{[t_k,t_{k+1}[}(u)\right)1_{[t_k,t_{k+1}[}(t).
\end{eqnarray}
It is straightforward to see that $0\leq K_t^{\pi}(u)\leq 1$ for all $t,u\in [0,T]$ and 
\begin{eqnarray*}
\dot{B_t^{\pi}}=\frac{B_{t_{k+1}}-B_{t_k}}{t_{k+1}-t_k}\quad\mbox{if }t\in [t_{k},t_{k+1}[.
\end{eqnarray*}
Moreover, in analogy with (\ref{B smooth}) we set
\begin{eqnarray}\label{def K dot}
\dot{B}^{\pi}_t=\int_0^T\partial_tK_t^{\pi}(u)dB_u\quad\mbox{ with }\quad\partial_tK_t^{\pi}(u)=\sum_{k=0}^{n-1}\frac{1}{t_{k+1}-t_k}1_{[t_k,t_{k+1}[}(u)1_{[t_k,t_{k+1}[}(t).
\end{eqnarray}

\subsection{Existence of the density}\label{subsection 1}

The aim of the present section is to prove the first statement in Theorem \ref{main theorem}, i.e. the absolute continuity of the law of $X_t^{\pi}$ with respect to the one dimensional Lebesgue measure. We will assume that the function $\sigma$ in the diffusion coefficient of equation (\ref{SDE}) is identically equal to one. The general case can be recovered with straightforward modifications. \\
As a first step we investigate the Malliavin regularity of $X_t^{\pi}$ and write an  equation for $DX_t^{\pi}$.
\begin{theorem}\label{theorem on DX}
For any $t\in [s,T]$ the random variable $X_t^{\pi}$ belongs to $\mathbb{D}^{2,p}$ for all $p\geq 1$. Moreover,
\begin{eqnarray}\label{equation for DX}
D_uX_t^{\pi}&=&D_uY+\int_s^tb_x(r,X_r^{\pi})D_uX_r^{\pi}dr+\int_s^tD_uX_r^{\pi}\diamond\dot{B_r^{\pi}}dr\\
&&+\int_s^tX_r^{\pi}\partial_r{K}^{\pi}_{r}(u)dr\nonumber
\end{eqnarray}
\end{theorem}

\begin{remark}
The need for checking that $X_t^{\pi}$ belongs to $\mathbb{D}^{2,p}$ in order to write the equation for $D_uX_t^{\pi}$ is due to the fact that integrand in
\begin{eqnarray*}
\int_s^tD_uX_r^{\pi}\diamond\dot{B_r^{\pi}}dr
\end{eqnarray*}
requires a control on the second order Malliavin derivative of $X_r^{\pi}$ to be well defined. In fact, from (\ref{definition Wick})  we have
\begin{eqnarray*}
D_uX_r^{\pi}\diamond\dot{B_r^{\pi}}=\delta(D_uX_r^{\pi}\partial_rK_r^{\pi}(\cdot))
\end{eqnarray*} 
implying that $D_uX_r^{\pi}\in\mathbb{D}^{1,2}$ is a sufficient conditions for the membership of $D_uX_r^{\pi}\partial_rK_r^{\pi}(\cdot)$ to the domain of $\delta$ (see Proposition 1.3.1 in \cite{Nualart}). 
\end{remark}

\begin{proof}
Employing a standard reduction method in combination with the properties of the Wick product, we can represent the solution $\{X_t^{\pi}\}_{t\in [s,T]}$ of equation (\ref{SDE}) as $X_t^{\pi}=Z_t^{\pi}\diamond\mathcal{E}(K_{s,t}^{\pi})$ where $K_{s,t}^{\pi}(\cdot):=K_t^{\pi}(\cdot)-K_s^{\pi}(\cdot)$ and  $\{Z_t^{\pi}\}_{t\in [s,T]}$ is the unique solution of the equation
\begin{eqnarray}
\left\{ \begin{array}{ll}\label{equation for Z}
\dot{Z}_t^{\pi}=b(t,Z_t^{\pi}\diamond\mathcal{E}(K_{s,t}^{\pi}))\diamond \mathcal{E}(-K_{s,t}^{\pi})\\
Z_s^{\pi}=Y
\end{array}\right.
\end{eqnarray}
(see \cite{BL} for the details of this technique). Moreover, according to the Gjessing formula (\ref{Gjessing}) we have
\begin{eqnarray*}
Z_t^{\pi}\diamond\mathcal{E}(K_{s,t}^{\pi})=\left(\mathcal{T}_{-K_{s,t}^{\pi}}Z_t^{\pi}\right)\cdot\mathcal{E}(K_{s,t}^{\pi}).
\end{eqnarray*}
Therefore, since $\mathcal{E}(K_{s,t}^{\pi})$ belongs to $\mathbb{D}^{k,p}$ for all $p\geq 1$ and $k\in\mathbb{N}$ and $\mathcal{T}_{-K_{s,t}^{\pi}}$ maps  $\mathbb{D}^{k,p}$ into $\mathbb{D}^{k,q}$ for any $q<p$ and $k\in\mathbb{N}$, the first part of the statement will follow from Proposition 1.5.6 in \cite{Nualart} once we prove that  $Z_t^{\pi}$ belongs to $\mathbb{D}^{2,p}$ for all $p\geq 1$.\\
Resorting once more to the Gjessing formula we can rewrite equation (\ref{equation for Z}) as
\begin{eqnarray}\label{equation for Z bis}
Z_t^{\pi}=Y+\int_s^tb(r,Z_r^{\pi}\cdot\mathcal{E}(-K_{s,r}^{\pi})^{-1})\cdot\mathcal{E}(-K_{s,r}^{\pi})dr.
\end{eqnarray}
We first derive an upper bound for $|Z_t^{\pi}|$ which will be useful to control the norms in $\mathbb{D}^{2,p}$:
\begin{eqnarray*}
|Z_t^{\pi}|&\leq&|Y|+\int_s^t|b(r,Z_r^{\pi}\cdot\mathcal{E}(-K_{s,r}^{\pi})^{-1})|\cdot\mathcal{E}(-K_{s,r}^{\pi})dr\\
&\leq&|Y|+M\int_s^t(1+|Z_r^{\pi}|\cdot\mathcal{E}(-K_{s,r}^{\pi})^{-1})\cdot\mathcal{E}(-K_{s,r}^{\pi})dr\\
&=&|Y|+M\int_s^t\mathcal{E}(-K_{s,r}^{\pi})dr+M\int_s^t|Z_r^{\pi}|dr\\
&\leq&|Y|+M\int_s^T\mathcal{E}(-K_{s,r}^{\pi})dr+M\int_s^t|Z_r^{\pi}|dr.
\end{eqnarray*} 
The positive constant $M$ utilized above comes from the inequality
\begin{eqnarray*}
|b(t,x)|\leq M(1+|x|),\quad t\in [0,T], x\in\mathbb{R}
\end{eqnarray*}
which follows from the assumptions on $b$ (in particular the boundedness of the first partial derivative with respect to $x$). By the Gronwall inequality we deduce that
\begin{eqnarray*}
|Z_t^{\pi}|&\leq&\left(|Y|+M\int_s^T\mathcal{E}(-K_{s,r}^{\pi})dr\right)e^{M(t-s)}
\end{eqnarray*}
and hence
\begin{eqnarray*}
\sup_{t\in [s,T]}|Z_t^{\pi}|&\leq&\left(|Y|+M\int_s^T\mathcal{E}(-K_{s,r}^{\pi})dr\right)e^{M(T-s)}.
\end{eqnarray*}
Computing the $L^p(\Omega)$-norms we get
\begin{eqnarray}\label{bound for Z}
\left\Vert\sup_{t\in [s,T]}|Z_t^{\pi}|\right\Vert_p&\leq&\left(\Vert Y\Vert_p+M\int_s^T\Vert\mathcal{E}(-K_{s,r}^{\pi})\Vert_pdr\right)e^{M(T-s)}\nonumber\\
&=&\left(\Vert Y\Vert_p+M\int_s^T\exp\{(p-1)|K_{s,r}^{\pi}|^2_H/2\}dr\right)e^{M(T-s)}\nonumber\\
&\leq&\left(\Vert Y\Vert_p+C\right)e^{M(T-s)}
\end{eqnarray}
where $C$ is a positive constant depending on $p$, $T$ and $M$ (recall that $0\leq K_t^{\pi}(u)\leq 1$ and hence that $|K_{s,t}^{\pi}|\leq 1$). We are now ready to compute the Malliavin derivative of $Z_t^{\pi}$; from equation (\ref{equation for Z bis}) we obtain
\begin{eqnarray*}
D_uZ_t^{\pi}&=&D_uY+\int_s^tD_u[b(r,Z_r^{\pi}\cdot\mathcal{E}(-K_{s,r}^{\pi})^{-1})\cdot\mathcal{E}(-K_{s,r}^{\pi})]dr\\
&=&D_uY+\int_s^tb_x(r,Z_r^{\pi}\cdot\mathcal{E}(-K_{s,r}^{\pi})^{-1})(D_uZ_r^{\pi}+Z_r^{\pi}K_{s,r}^{\pi}(u))dr\\
&&-\int_s^tb(r,Z_r^{\pi}\cdot\mathcal{E}(-K_{s,r}^{\pi})^{-1})\cdot\mathcal{E}(-K_{s,r}^{\pi})K_{s,r}^{\pi}(u)dr
\end{eqnarray*}
which gives
\begin{eqnarray*}
|D_uZ_t^{\pi}|&\leq&|D_uY|+L\int_s^t|D_uZ_r^{\pi}|dr+L\int_s^t|Z_r^{\pi}||K_{s,r}^{\pi}(u)|dr\\
&&+M\int_s^t(1+|Z_r^{\pi}\cdot\mathcal{E}(-K_{s,r}^{\pi})^{-1}|)\cdot\mathcal{E}(-K_{s,r}^{\pi})|K_{s,r}^{\pi}(u)|dr\\
&\leq&|D_uY|+L\int_s^t|D_uZ_r^{\pi}|dr+L\int_s^t|Z_r^{\pi}|dr\\
&&+M\int_s^t\mathcal{E}(-K_{s,r}^{\pi})dr+M\int_s^t|Z_r^{\pi}|dr\\
&\leq&|D_uY|+L\int_s^t|D_uZ_r^{\pi}|dr+(L+M)\int_s^T|Z_r^{\pi}|dr\\
&&+M\int_s^T\mathcal{E}(-K_{s,r}^{\pi})dr.
\end{eqnarray*}
Here, $L$ stands for the Lipschitz constant of $b$ with respect to the variable $x$. Then, by the Gronwall inequality we obtain
\begin{eqnarray*}
|D_uZ_t^{\pi}|\leq \left(|D_uY|+(L+M)\int_s^T|Z_r^{\pi}|dr+M\int_s^T\mathcal{E}(-K_{s,r}^{\pi})dr\right)e^{L(t-s)}
\end{eqnarray*}
which in turn implies
\begin{eqnarray*}
\vert DZ_t^{\pi}\vert_{H}&\leq& e^{L(t-s)}\left(\vert DY\vert_H+\sqrt{T}(L+M)\int_s^T|Z_r^{\pi}|dr+\sqrt{T}M\int_s^T\mathcal{E}(-K_{s,r}^{\pi})dr\right)\\
&\leq&e^{L(t-s)}\left(\vert DY\vert_H+\sqrt{T}(L+M)(T-s)\sup_{r\in [s,T]}|Z_r^{\pi}|+\sqrt{T}M\int_s^T\mathcal{E}(-K_{s,r}^{\pi})dr\right).
\end{eqnarray*}
With the help of estimate (\ref{bound for Z}) we can conclude that
\begin{eqnarray*}
\Vert\vert DZ_t^{\pi}\vert_{H}\Vert_p\leq C_1\Vert Y\Vert_{\mathbb{D}^{1,p}}+C_2
\end{eqnarray*} 
where $C_1$ and $C_2$  are positive constants depending on $p$, $T$, $M$ and $L$. This proves that $Z_t^{\pi}$ belongs to $\mathbb{D}^{1,p}$. For the second order Malliavin derivative of $Z_t^{\pi}$ one proceeds as before differentiating twice the identity (\ref{equation for Z bis}) and resorting to the Grownwall inequality for the estimation of the  norm of $D_vD_uZ_t^{\pi}$. \\
To get identity (\ref{equation for DX}) one has simply to differentiate the equality
\begin{eqnarray*}
X_t^{\pi}=Y+\int_s^tb(r,X_r^{\pi})dr+\int_s^tX_r^{\pi}\diamond\dot{B}_r^{\pi}dr
\end{eqnarray*}
in connection with the following chain rule for the Wick product:
\begin{eqnarray*}
D_u(X_r^{\pi}\diamond\dot{B}_r^{\pi})&=&(D_uX_r^{\pi})\diamond\dot{B}_r^{\pi}+X_r^{\pi}\diamond D_u\dot{B}_r^{\pi}\\
&=&(D_uX_r^{\pi})\diamond\dot{B}_r^{\pi}+X_r^{\pi}\diamond \partial_rK_r^{\pi}(u)\\
&=&(D_uX_r^{\pi})\diamond\dot{B}_r^{\pi}+X_r^{\pi}\cdot\partial_rK_r^{\pi}(u)
\end{eqnarray*}
The proof is complete.
\end{proof}

\noindent Equation (\ref{equation for DX}) shows that $D_uX_r^{\pi}$ solves a linear differential equation containing the non homogeneous term
\begin{eqnarray*}
\int_s^tX_r^{\pi}\partial_r{K}^{\pi}_{r}(u)dr.
\end{eqnarray*}
In the next proposition we will find a direction $h^{\pi}\in H$ along which the quantity above will be identically zero and this will produce an explicit expression for the resulting directional derivative of $X_r^{\pi}$.

\begin{proposition}\label{direction}
There exists a function $h^{\pi}\in H$ such that $h^{\pi}\neq 0$ almost everywhere and $\int_0^T\partial_rK^{\pi}_{r}(u)h^{\pi}(u)du=0$ for all $r\in [0,T]$. Moreover, 
\begin{eqnarray}\label{D_hX}
D_{h^{\pi}}X_t^{\pi}=\mathcal{T}_{-K_{s,t}^{\pi}}D_{h^{\pi}}Y\cdot\exp\left\{\int_s^tb_x(r,\mathcal{T}_{-K_{s,t}^{\pi}}\mathcal{T}_{K_{s,r}^{\pi}}X_r^{\pi})dr\right\}\cdot\mathcal{E}(K_{s,t}^{\pi}).
\end{eqnarray}
\end{proposition}

\begin{proof}
Let $h^{\pi}$ be a function in $H$ such that $h^{\pi}\neq 0$ almost everywhere and
\begin{eqnarray*}
\frac{1}{t_{k+1}-t_k}\int_{t_k}^{t_{k+1}}h^{\pi}(u)du=0\quad\mbox{ for all }k\in\{0,..., n-1\}.
\end{eqnarray*}
Then, according to the second equation in (\ref{def K dot}) we have
\begin{eqnarray*}
\int_0^T\partial_rK^{\pi}_{r}(u)h^{\pi}(u)du&=&\sum_{k=0}^{n-1}\frac{1}{t_{k+1}-t_k}\int_{t_k}^{t_{k+1}}h^{\pi}(u)du\cdot1_{[t_k,t_{k+1}[}(r)\\
&=&0\quad\mbox{ for all }r\in [0,T].
\end{eqnarray*} 
If now we multiply both sides of equation (\ref{equation for DX}) by the function $h^{\pi}$ and integrate with respect to $u$ between zero and $T$, we see that the last term is zero and we are left with 
\begin{eqnarray}\label{equation for DX bis}
D_{h^{\pi}}X_t^{\pi}=D_{h^{\pi}}Y+\int_s^tb_x(r,X_r^{\pi})D_{h^{\pi}}X_r^{\pi}dr+\int_s^tD_{h^{\pi}}X_r^{\pi}\diamond\dot{B_r^{\pi}}dr.
\end{eqnarray}
Equation (\ref{equation for DX bis}) is linear and homogeneous in $D_{h^{\pi}}X_{\cdot}^{\pi}$. We now find its unique solution. First of all observe that employing the reduction method mentioned above we can write
\begin{eqnarray*}
D_{h^{\pi}}X_t^{\pi}=V^{\pi}_t\diamond\mathcal{E}(K^{\pi}_{s,t})
\end{eqnarray*} 
where $\{V_t^{\pi}\}_{t\in [s,T]}$ satisfies
\begin{eqnarray}\label{C}
V_t^{\pi}=D_{h^{\pi}}Y+\int_s^t(b_x(r,X_r^{\pi})(V_r^{\pi}\diamond\mathcal{E}(K^{\pi}_{s,r})))\diamond\mathcal{E}(-K^{\pi}_{s,r})dr
\end{eqnarray}
We note that applying the Gjessing formula (\ref{Gjessing}) twice we can write
\begin{eqnarray*}
(b_x(r,X_r^{\pi})(V_r^{\pi}\diamond\mathcal{E}(K^{\pi}_{s,r})))\diamond\mathcal{E}(-K^{\pi}_{s,r})&=&
(b_x(r,X_r^{\pi})\cdot\mathcal{T}_{-K^{\pi}_{s,r}}V_r^{\pi}\cdot\mathcal{E}(K^{\pi}_{s,r}))\diamond\mathcal{E}(-K^{\pi}_{s,r})\\
&=&\mathcal{T}_{K^{\pi}_{s,r}}(b_x(r,X_r^{\pi})\cdot\mathcal{T}_{-K^{\pi}_{s,r}}V_r^{\pi}\cdot\mathcal{E}(K^{\pi}_{s,r}))\cdot\mathcal{E}(-K^{\pi}_{s,r})\\
&=&b_x(r,\mathcal{T}_{K^{\pi}_{s,r}}X_r^{\pi})\cdot V_r^{\pi}\cdot\mathcal{T}_{K^{\pi}_{s,r}}\mathcal{E}(K^{\pi}_{s,r})\cdot\mathcal{E}(-K^{\pi}_{s,r})\\
&=&b_x(r,\mathcal{T}_{K^{\pi}_{s,r}}X_r^{\pi})\cdot V_r^{\pi}.
\end{eqnarray*}
Therefore, equation (\ref{C}) now reads
\begin{eqnarray*}
V_t^{\pi}&=&D_{h^{\pi}}Y+\int_s^t(b_x(r,X_r^{\pi})(V_r^{\pi}\diamond\mathcal{E}(K^{\pi}_{s,r})))\diamond\mathcal{E}(-K^{\pi}_{s,r})dr\\
&=&D_{h^{\pi}}Y+\int_s^tb_x(r,\mathcal{T}_{K^{\pi}_{s,r}}X_r^{\pi})\cdot V_r^{\pi}dr
\end{eqnarray*}
 implying that
\begin{eqnarray*}
V_t^{\pi}=D_{h^{\pi}}Y\exp\left\{\int_s^tb_x(r,\mathcal{T}_{K^{\pi}_{s,r}}X_r^{\pi})dr\right\}.
\end{eqnarray*}
Therefore,
\begin{eqnarray*}
D_{h^{\pi}}X_t^{\pi}&=&V^{\pi}_t\diamond\mathcal{E}(K^{\pi}_{s,t})\\
&=&\left(D_{h^{\pi}}Y\exp\left\{\int_s^tb_x(r,\mathcal{T}_{K^{\pi}_{s,r}}X_r^{\pi})dr\right\}\right)\diamond\mathcal{E}(K^{\pi}_{s,t})\\
&=&\mathcal{T}_{-K^{\pi}_{s,t}}D_{h^{\pi}}Y\cdot\mathcal{T}_{-K^{\pi}_{s,t}}\exp\left\{\int_s^tb_x(r,\mathcal{T}_{K^{\pi}_{s,r}}X_r^{\pi})dr\right\}\cdot\mathcal{E}(K^{\pi}_{s,t})\\
&=&\mathcal{T}_{-K^{\pi}_{s,t}}D_{h^{\pi}}Y\cdot\exp\left\{\int_s^tb_x(r,\mathcal{T}_{-K^{\pi}_{s,t}}\mathcal{T}_{K^{\pi}_{s,r}}X_r^{\pi})dr\right\}\cdot\mathcal{E}(K^{\pi}_{s,t}).
\end{eqnarray*}
The proof is complete.
\end{proof}

\noindent We are now ready to prove that the law of $X_t^{\pi}$ is absolutely continuous with respect to the Lebesgue measure. Our strategy is based on the following criteria that generalizes to some extent Proposition 2.1.1 in \cite{Nualart}:\\

\noindent \emph{Suppose that $X\in \mathbb{D}^{1,2}$ and let $h\in H$ be such that $D_hX\neq 0$ almost surely and $h/D_hX$ belongs to the domain of $\delta$. Then, $X$ possesses a bounded and continuous density.}\\

\noindent First of all we note that the assumption 
\begin{eqnarray*}
E[|D_{h^{\pi}}Y|^{-q}]\mbox{ is finite for some $q>4$}
\end{eqnarray*}
from Theorem \ref{main theorem} implies that $D_{h^{\pi}}Y\neq 0$ almost surely (and also that $\mathcal{T}_{-K^{\pi}_{s,t}}D_{h^{\pi}}Y\neq 0$ since Cameron-Martin shifts preserve negligible sets). This fact combined with Proposition \ref{direction} entails that $D_{h^{\pi}}X_t^{\pi}\neq 0$ almost surely. Hence, we are left with the verification that $h^{\pi}/D_{h^{\pi}}X_t^{\pi}$ belongs to the domain of $\delta$; a sufficient condition for that is $h^{\pi}/D_{h^{\pi}}X_t^{\pi}\in\mathbb{D}^{1,2}(H)$.  \\
We observe that
\begin{eqnarray*}
E[\vert h^{\pi}/D_{h^{\pi}}X_t^{\pi}\vert^2_H]&=&\vert h^{\pi}\vert^2_HE[1/|D_{h^{\pi}}X_t^{\pi}|^2]
\end{eqnarray*}
and
\begin{eqnarray*}
E[\vert D(h^{\pi}/D_{h^{\pi}}X_t^{\pi})\vert^2_{H\otimes H}]&=&\vert h^{\pi}\vert^2_HE[\vert DD_{h^{\pi}}X_t^{\pi}\vert^2_H/|D_{h^{\pi}}X_t^{\pi}|^4]\\
&\leq&\vert h^{\pi}\vert^2_HE[\vert DD_{h^{\pi}}X_t^{\pi}\vert^{2p}_H]^{1/p}\cdot E[1/|D_{h^{\pi}}X_t^{\pi}|^{4q}]^{1/q}
\end{eqnarray*}
where we applied the H\"older inequality with $1/p+1/q=1$. Therefore, $h^{\pi}/D_{h^{\pi}}X_t^{\pi}\in\mathbb{D}^{1,2}(H)$ if $D_{h^{\pi}}X_t^{\pi}\in\mathbb{D}^{1,p}$ for all $p\geq 1$ and 
\begin{eqnarray}\label{inverse integrability}
E[1/|D_{h^{\pi}}X_t^{\pi}|^{q}]\mbox{ is finite for some $q>4$}.
\end{eqnarray}
The first condition follows from the fact that $X_t^{\pi}\in\mathbb{D}^{2,p}$ (see Theorem \ref{theorem on DX}). We now verify the validity of (\ref{inverse integrability}). Let $q>4$ be such that $E[|D_{h^{\pi}}Y|^{-q}]$ is finite and let $\alpha>1$ and $\varepsilon>0$ be such that $\tilde{q}:=(q-\varepsilon)/\alpha>4$; then, recalling the identity (\ref{D_hX}) we can write
\begin{eqnarray*}
|D_{h^{\pi}}X_t^{\pi}|^{-\tilde{q}}&=&|\mathcal{T}_{-K_{s,t}^{\pi}}D_{h^{\pi}}Y|^{-\tilde{q}}\cdot\exp\left\{-\tilde{q}\int_s^tb_x(r,\mathcal{T}_{-K_{s,t}^{\pi}}\mathcal{T}_{K_{s,r}^{\pi}}X_r^{\pi})dr\right\}\cdot|\mathcal{E}(K_{s,t}^{\pi})|^{-\tilde{q}}\\
&\leq&e^{\tilde{q}L(t-s)}|\mathcal{T}_{-K_{s,t}^{\pi}}D_{h^{\pi}}Y|^{-\tilde{q}}\cdot|\mathcal{E}(K_{s,t}^{\pi})|^{-\tilde{q}}\\
&=&e^{\tilde{q}L(t-s)}\mathcal{T}_{-K_{s,t}^{\pi}}|D_{h^{\pi}}Y|^{-\tilde{q}}\cdot|\mathcal{E}(K_{s,t}^{\pi})|^{-\tilde{q}}.
\end{eqnarray*}
Therefore,
\begin{eqnarray*}
E[|D_{h^{\pi}}X_t^{\pi}|^{-\tilde{q}}]&\leq& e^{\tilde{q}L(t-s)}E[\mathcal{T}_{-K_{s,t}^{\pi}}|D_{h^{\pi}}Y|^{-\tilde{q}}\cdot|\mathcal{E}(K_{s,t}^{\pi})|^{-\tilde{q}}]\\
&\leq& e^{\tilde{q}L(t-s)}E[\mathcal{T}_{-K_{s,t}^{\pi}}|D_{h^{\pi}}Y|^{-q+\varepsilon}]^{1/\alpha}\cdot E[|\mathcal{E}(K_{s,t}^{\pi})|^{-\tilde{q}\beta}]^{1/\beta}\\
&\leq&CE[|D_{h^{\pi}}Y|^{-q}]^{(q-\varepsilon)/q\alpha}
\end{eqnarray*}
where $1/\alpha+1/\beta=1$ and $C$ is a positive constant depending on $L$, $\alpha$, $\varepsilon$, $q$ and $T$. In the last inequality we utilized the fact that $\mathcal{T}_{-K_{s,t}^{\pi}}$ maps $L^q(\Omega)$ into $L^r(\Omega)$ for all $r<p$ and the membership of $\mathcal{E}(K_{s,t}^{\pi})$ to all the spaces $L^p(\Omega)$ for $p\geq 1$. The assumption on $Y$ completes the proof of the claim (\ref{inverse integrability}) which in turn implies the absolute continuity of the law of $X_t^{\pi}$ with respect to the one dimensional Lebesgue measure with a bounded and continuous density.

\subsection{Fokker-Planck equation}\label{subsection 2}

\noindent We now prove the second part of Theorem \ref{main theorem}. We will assume as before that the function $\sigma$ in the diffusion coefficient of equation (\ref{SDE}) is identically equal to one.  \\
As for exact solutions, the key ingredient to relate the law of the solution of the stochastic equation (\ref{SDE}) with a Fokker-Planck type equation is the It\^o formula.
\begin{theorem}[It\^o formula]\label{Ito formula}
Let $\{X_t^{\pi}\}_{t\in [s,T]}$ be the unique solution of equation (\ref{SDE}) and let $\varphi\in C^{1,2}([0,T]\times\mathbb{R})$ be such that $\partial_x\varphi$ and $\partial_{xx}\varphi$ have at most polynomial growth at infinity. Then, for $s\leq t\leq T$ we have 
\begin{eqnarray*}
\varphi(t,X_t^{\pi})-\varphi(s,Y)&=&\int_s^t\left[\partial_t\varphi(r,X_r^{\pi})+\partial_x\varphi(r,X_r^{\pi})b(r,X_r^{\pi})\right]dr\\
&&+\int_s^t\partial_{xx}\varphi(r,X_r^{\pi})X_r^{\pi}D_{\partial_rK_r^{\pi}}X_r^{\pi}dr\\
&&+\int_s^t\left(\partial_{x}\varphi(r,X_r^{\pi})X_r^{\pi}\right)\diamond\dot{B}_r^{\pi}dr.
\end{eqnarray*}
\end{theorem}

\begin{proof}
First of all we observe that the condition on the growth at infinity of $\partial_x\varphi$ and $\partial_{xx}\varphi$ together with the membership of $X_r^{\pi}$ to $\mathbb{D}^{2,p}$ for all $p\geq 1$ ensure that
\begin{eqnarray*}
\left(\partial_{x}\varphi(r,X_r^{\pi})X_r^{\pi}\right)\diamond\dot{B}_r^{\pi}=\delta(\partial_{x}\varphi(r,X_r^{\pi})X_r^{\pi}\partial_rK_r^{\pi})
\end{eqnarray*}
is well defined since $\partial_{x}\varphi(r,X_r^{\pi})X_r^{\pi}\in\mathbb{D}^{1,2}$. Now, using equation (\ref{SDE}) we get
\begin{eqnarray}\label{last}
\varphi(t,X_t^{\pi})-\varphi(s,Y)&=&\int_s^t\left[\partial_t\varphi(r,X_r^{\pi})+\partial_x\varphi(r,X_r^{\pi})\dot{X}_r^{\pi}\right]dr\nonumber\\
&=&\int_s^t\left[\partial_t\varphi(r,X_r^{\pi})+\partial_x\varphi(r,X_r^{\pi})b(r,X_r^{\pi})\right]dr\nonumber\\
&&+\int_s^t\partial_x\varphi(r,X_r^{\pi})\cdot(X_r^{\pi}\diamond \dot{B}_r^{\pi})dr.
\end{eqnarray}
Moreover, according to Proposition 1.3.3 in \cite{Nualart} we can write
\begin{eqnarray*}
\partial_x\varphi(r,X_r^{\pi})\cdot(X_r^{\pi}\diamond \dot{B}_r^{\pi})&=&\partial_x\varphi(r,X_r^{\pi})\cdot\delta(X_r^{\pi}\partial_rK_r^{\pi})\\
&=&\delta(\partial_x\varphi(r,X_r^{\pi})X_r^{\pi}\partial_rK_r^{\pi})\\
&&+\int_0^T(D_u\partial_x\varphi(r,X_r^{\pi}))X_r^{\pi}\partial_rK_r^{\pi}(u)du\\
&=&(\partial_x\varphi(r,X_r^{\pi})X_r^{\pi})\diamond\dot{B}_r^{\pi}\\
&&+\int_0^T\partial_{xx}\varphi(r,X_r^{\pi})(D_uX_r^{\pi})X_r^{\pi}\partial_rK_r^{\pi}(u)du\\
&=&(\partial_x\varphi(r,X_r^{\pi})X_r^{\pi})\diamond\dot{B}_r^{\pi}\\
&&+\partial_{xx}\varphi(r,X_r^{\pi})X_r^{\pi}D_{\partial_rK_r^{\pi}}X_r^{\pi}.
\end{eqnarray*}
The substitution of the integrand in (\ref{last}) with the last member of the previous chain of equalities provides the desired formula.  
\end{proof}
\noindent Now, let $\varphi\in C_0^{\infty}([s,T]\times\mathbb{R})$; then, by Theorem \ref{Ito formula} we obtain
\begin{eqnarray*}
0&=&\varphi(T,X_T^{\pi})-\varphi(s,Y)\\
&=&\int_s^T\left[\partial_t\varphi(r,X_r^{\pi})+\partial_x\varphi(r,X_r^{\pi})b(r,X_r^{\pi})\right]dr\\
&&+\int_s^T\partial_{xx}\varphi(r,X_r^{\pi})X_r^{\pi}D_{\partial_rK_r^{\pi}}X_r^{\pi}dr\\
&&+\int_s^T\left(\partial_{x}\varphi(r,X_r^{\pi})X_r^{\pi}\right)\diamond\dot{B}_r^{\pi}dr.
\end{eqnarray*}
Taking the expectation, recalling that for all $X\in\mathbb{D}^{1,2}$
\begin{eqnarray*}
E[X\diamond\dot{B}_r^{\pi}]=E[\delta(X\partial_rK_r^{\pi})]=0
\end{eqnarray*} 
and denoting by $p^{\pi}(r,x)$ the density of the random variable $X_r^{\pi}$, we get
\begin{eqnarray*}
0&=&E\left[\int_s^T\partial_t\varphi(r,X_r^{\pi})+\partial_x\varphi(r,X_r^{\pi})b(r,X_r^{\pi})dr\right]\\
&&+E\left[\int_s^T\partial_{xx}\varphi(r,X_r^{\pi})X_r^{\pi}D_{\partial_rK_r^{\pi}}X_r^{\pi}dr\right]\\
&=&\int_s^TE\left[\partial_t\varphi(r,X_r^{\pi})+\partial_x\varphi(r,X_r^{\pi})b(r,X_r^{\pi})\right]dr\\
&&+\int_s^TE\left[\partial_{xx}\varphi(r,X_r^{\pi})X_r^{\pi}D_{\partial_rK_r^{\pi}}X_r^{\pi}\right]dr\\
&=&\int_s^T\int_{\mathbb{R}}(\partial_t\varphi(r,x)+\partial_x\varphi(r,x)b(r,x))p^{\pi}(r,x)dxdr\\
&&+\int_s^TE\left[\partial_{xx}\varphi(r,X_r^{\pi})X_r^{\pi}E[D_{\partial_rK_r^{\pi}}X_r^{\pi}|\sigma(X_r^{\pi})]\right]dr\\
&=&\int_s^T\int_{\mathbb{R}}(\partial_t\varphi(r,x)+\partial_x\varphi(r,x)b(r,x))p^{\pi}(r,x)dxdr\\
&&+\int_s^T\int_{\mathbb{R}}\partial_{xx}\varphi(r,x)xg(r,x)p^{\pi}(r,x)dxdr
\end{eqnarray*}
where $g:[s,T]\times\mathbb{R}\to\mathbb{R}$ is a measurable function such that
\begin{eqnarray}\label{definition of g}
g(r,x)=E[D_{\partial_rK_r^{\pi}}X_r^{\pi}|\sigma(X_r^{\pi})]|_{X_r^{\pi}=x}.
\end{eqnarray}
We observe that $g$ is uniquely defined up to sets that are negligible with respect to the law of $X_r^{\pi}$ and hence negligible with respect to the Lesbegue measure. According to this construction the function $g$ is barely measurable with only some integrability properties against the density $p^{\pi}(t,x)$. In fact, since $D_{\partial_rK_r^{\pi}}X_r^{\pi}\in L^{q}(\Omega)$ for all $q\geq 1$ we get the bound
\begin{eqnarray*}
\int_{\mathbb{R}}|g(r,x)|^qp^{\pi}(r,x)dx&=&E\left[|E[D_{\partial_rK_r^{\pi}}X_r^{\pi}|\sigma(X_r^{\pi})]|^q\right]\\
&\leq&E\left[|D_{\partial_rK_r^{\pi}}X_r^{\pi}|^q\right].
\end{eqnarray*}
We have therefore proved for any $\varphi\in C_0^{\infty}([s,T]\times\mathbb{R})$ the identity
\begin{eqnarray*}
\int_s^T\int_{\mathbb{R}}(\partial_t\varphi(r,x)+\partial_x\varphi(r,x)b(r,x)+\partial_{xx}\varphi(r,x)xg(r,x))p^{\pi}(r,x)dxdr=0
\end{eqnarray*}
which is equivalent to the statement that $p^{\pi}(t,x)$ is a distributional solution of the equation
\begin{eqnarray*}
(\partial_t+b(t,x)\partial_x+xg(t,x)\partial_{xx})^{\ast}u(t,x)=0.
\end{eqnarray*}

\begin{example}
Consider the case where $b(t,x)=0$ and $\sigma(t)=1$. Then, equation (\ref{SDE}) reads 
\begin{eqnarray}\label{example SDE}
\left\{ \begin{array}{ll}
\dot{X}_t^{\pi}=X_t^{\pi}\diamond \dot{B}_t^{\pi},\quad t\in]s,T] \\
X_s^{\pi}= Y
\end{array}\right.
\end{eqnarray}
The solution to this equation is given by the formula
\begin{eqnarray}\label{last 2}
X_t^{\pi}=Y\diamond\mathcal{E}(K^{\pi}_{s,t}).
\end{eqnarray}
If we take $Y=X_s$ where $\{X_t\}_{t\in [0,T]}$ is the solution of
\begin{eqnarray}
\left\{ \begin{array}{ll}\label{exact SDE example}
dX_t=X_tdB_t,\quad t\in]0,T] \\
X_0= x\in\mathbb{R}
\end{array}\right.
\end{eqnarray}
(i.e. the SDE we are approximating with (\ref{example SDE})) then $X_s=x\cdot\mathcal{E}(1_{[0.s[})$ and by the properties of the Wick product we can write (\ref{last 2}) as
\begin{eqnarray*}
X_t^{\pi}&=&X_s\diamond\mathcal{E}(K^{\pi}_{s,t})\\
&=&x\cdot\mathcal{E}(1_{[0.s[})\diamond\mathcal{E}(K^{\pi}_{s,t})\\
&=&x\cdot\mathcal{E}(1_{[0.s[}+K^{\pi}_{s,t}).
\end{eqnarray*}
We now aim at finding an explicit expression for $g$ from formula (\ref{definition of g}) and hence for the Fokker-Planck-type equation associated to (\ref{example SDE}). First of all we compute the Malliavin derivative of $X_t^{\pi}$:
\begin{eqnarray*}
D_uX_t^{\pi}&=&D_u(x\cdot\mathcal{E}(1_{[0.s[}+K^{\pi}_{s,t}))\\
&=&X_t^{\pi}\cdot(1_{[0.s[}(u)+K^{\pi}_{s,t}(u)).
\end{eqnarray*}
Then, we get
\begin{eqnarray*}
g(t,x)&=&E[D_{\partial_tK_t^{\pi}}X_t^{\pi}|\sigma(X_t^{\pi})]|_{X_t^{\pi}=x}\\
&=&E\left[X_t^{\pi}\int_0^T(1_{[0.s[}(u)+K^{\pi}_{s,t}(u))\partial_tK_t^{\pi}(u)du\Big|\sigma(X_t^{\pi})\right]\Big|_{X_t^{\pi}=x}\\
&=&x\int_0^T(1_{[0.s[}(u)+K^{\pi}_{s,t}(u))\partial_tK_t^{\pi}(u)du.
\end{eqnarray*} 
Denoting
\begin{eqnarray*}
\xi^{\pi}(t):=\int_0^T(1_{[0.s[}(u)+K^{\pi}_{s,t}(u))\partial_tK_t^{\pi}(u)du
\end{eqnarray*}
we can write the Fokker-Planck-type equation for (\ref{example SDE}) as
\begin{eqnarray*}
(\partial_t+x^2\xi^{\pi}(t)\partial_{xx})^{\ast}p^{\pi}(t,x)=0
\end{eqnarray*} 
to be compared with 
\begin{eqnarray*}
(\partial_t+(x^2/2)\partial_{xx})^{\ast}p(t,x)=0
\end{eqnarray*} 
which is the one for the exact equation (\ref{exact SDE example}).
\end{example}


\begin{thebibliography}{99}

\bibitem{BT}
V. Bally and D. Talay, The law of the Euler scheme for stochastic differential equations. II. Convergence rate of the density, \emph{Monte Carlo Methods Appl.} \textbf{2} (1996) 93-128.

\bibitem{BL}
B. K. Ben Ammou and A. Lanconelli, Rate of convergence for Wong-Zakai-type approximations of It\^o stochastic differential equations, \emph{J. Theor. Probab.} (2018) https://doi.org/10.1007/s10959-018-0837-x

\bibitem{Bogachev}
V. I. Bogachev, \emph{Gaussian Measures}, American Mathematical Society, Providence, 1998.

\bibitem{Brezniak Flandoli}
Z. Brezniak and F. Flandoli, Almost sure approximation of Wong-Zakai type for stochastic partial differential equations, \emph{Stoch. Proc. and their Appl.} \textbf{55} (1995) 329-358.

\bibitem{DLS}
P. Da Pelo, A. Lanconelli and A. I. Stan, An It\^o formula for a family of stochastic integrals and related Wong-Zakai theorems, \emph{Stoch. Proc. and their Appl.} \textbf{123} (2013) 3183-3200.

\bibitem{Gyongy Shmatkov}
I. Gy\"ongy and A. Shmatkov, Rate of Convergence of Wong–Zakai Approximations for Stochastic Partial Differential Equations, \emph{Appl. Math. Optim.} \textbf{54} (2006) 315-341.

\bibitem{Hairer Pardoux}
M. Hairer and E. Pardoux, A Wong–Zakai theorem for stochastic PDEs, \emph{J. Math. Soc. Japan.} \textbf{67} (2015) 1551-1604.

\bibitem{HOUZ}
H. Holden, B. {\O}ksendal, J. Ub{\o}e and T.-S. Zhang, \emph{Stochastic Partial Differential Equations - II Edition}, Springer, New York, 2010.

\bibitem{Hu}
Y. Hu, \emph{Analysis on Gaussian spaces}, World Scientific Publishing Co. Pte. Ltd., Hackensack, NJ, 2017.

\bibitem{HKX}
Y. Hu, G. Kallianpur and  J. Xiong, An approximation for Zakai equation, \emph{Applied Math. Optimiz.} \textbf{45} (2002) 23–44.

\bibitem{HN}
 Y. Hu and D. Nualart, Rough path analysis via fractional calculus, \emph{Trans. Amer. Math. Soc.} \textbf{361} (2009) 2689-2718.

\bibitem{HO}
Y. Hu and B. {\O}ksendal, Wick approximation of quasilinear stochastic differential equations, {\em Stochastic analysis and related topics} {\bf V} Birkh\"{a}user (1996) 203-231.

\bibitem{Janson}
S. Janson, \emph{Gaussian Hilbert spaces}, Cambridge Tracts in Mathematics, 129. Cambridge University Press, Cambridge, 1997.

\bibitem{KS}
I. Karatzas and S. E. Shreve, \emph{Brownian motion and stochastic calculus}, Springer-Verlag, New York, 1991.

\bibitem{Konecny}
F. Konecny, A Wong-Zakai approximation of stochastic differential equations, \emph{Journal of Multivariate Analysis} \textbf{13} (1983) 605-611.

\bibitem{L}
A. Lanconelli, Standardizing densities on Gaussian spaces, \emph{Statistics and Probability Letters} \textbf{137} (2018) 243-250

\bibitem{LS}
A. Lanconelli and A. I. Stan, A H\"older inequality for norms of Poissonian Wick products, \emph{Inf. Dim. Anal. Quantum Prob. Related Topics} \textbf{16} (2013) https://doi.org/10.1142/S0219025713500227

\bibitem{LS Bernoulli}
A. Lanconelli and A. I. Stan, A note on a local limit theorem for Wiener space valued random variables, \emph{Bernoulli} \textbf{22} (2016) 2001-2112.

\bibitem{Nualart}
D. Nualart, \emph{Malliavin calculus and Related Topics - II Edition}, Springer, New York,
2006.

\bibitem{Stroock Varadhan}
D. W. Stroock and S. R. S. Varadhan, On the support of diffusion processes with applications to the strong maximum principle, \emph{Proceedings $6$-th Berkeley Symposium Math. Statist. Probab.} \textbf{3} (1972) University of California Press, Berkeley, 333-359.

\bibitem{Tessitore Zabczyk}
G. Tessitore and G. J. Zabczyk, Wong-Zakai approximations of stochastic evolution equations, \emph{Journal Evol. Equ.} \textbf{6} (2006) 621-655.

\bibitem{WZ}
E. Wong and M. Zakai: On the relation between ordinary and stochastic differential equations, {\em Intern. J. Engr. Sci.} {\bf 3} (1965) 213-229.

\bibitem{WZ2}
E. Wong and M. Zakai, Riemann-Stieltjes approximations of stochastic integrals, \emph{Z. Wahrscheinlichkeitstheorie verw. Geb.} \textbf{12} (1969) 87-97.

\end{thebibliography}
\end{document}